\documentclass[10pt]{article}

 \usepackage{amsmath,amsthm,amssymb}
 \usepackage{makeidx,epsfig,lscape}
 \usepackage{color,colortbl}
 \usepackage{fancyhdr}
 \usepackage{times}
 \usepackage{xcolor,pict2e}
 \usepackage[latin1]{inputenc}
 
\newcommand{\R}{\mathbb R}

\newcommand{\F}{\mathcal F}
\newcommand{\Qv}{\mathbb Q}
\newcommand{\Pv}{\mathbb P}
\newcommand{\E}{\mathbb E}

 \renewcommand{\headrulewidth}{0pt}
 \renewcommand{\footrulewidth}{0.5pt}

 \definecolor{myaqua}{rgb}{0.0,0.5,0.55}
 \definecolor{lightaqua}{rgb}{0.75,0.95,0.95}

 \usepackage[colorlinks = true,
            linkcolor = myaqua,
            urlcolor  = blue,
            citecolor = red,
            pdfpagemode=UseOutlines,
            bookmarksnumbered=true,
            pdfpagelabels,
            breaklinks]{hyperref}

\usepackage{caption}
\usepackage{floatrow}

\newtheorem{theorem}{Theorem}
\newtheorem{prop}[theorem]{Proposition}

\newtheorem{coro}{Corollary}
\newtheorem{defn}{Definition}[section]
\newtheorem{rem}{Remark}[section]
\def\lin#1#2{\textcolor[rgb]{0.6,0.6,0.6}{\vspace*{#1mm} \hrule
   height 3 pt \vspace*{#2mm}}}
\def\bt{\begin{tabular}}
\def\et{\end{tabular}}
\def\and{\mbox{ and }}
\def\E{\mbox{\bf E}}

\def\1{{\bf 1}}

 \def\boxx#1#2#3#4#5{
 {\linethickness{#4pt}\put(#1,#5){\color{myaqua}{\line(1,0){#3}}}}
 \multiput(#1,#2)(0,#4){2}{\line(1,0){#3}}
 \multiput(#1,#2)(#3,0){2}{\line(0,1){#4}}
  }

\begin{document}


 $\mbox{ }$

 \vskip 12mm

{ 

{\noindent{\Large\bf\color{myaqua}
  Some contributions to the study of stochastic processes of the classes \texorpdfstring{$\Sigma(H)$}{sigma(H)} and \texorpdfstring{$(\Sigma)$}{sigma}  }} 
%
\\[6mm]
{\bf Fulgence EYI-OBIANG$^1$, Youssef OUKNINE$^2$, Octave MOUTSINGA$^3$, Gerald TRUTNAU$^4$}}
\\[2mm]
{ 
 $^1$URMI Laboratory, Département de Mathématiques et Informatique, Faculté des Sciences, Université des Sciences et Techniques de Masuku, Franceville, Gabon 
 \\
Email: \href{mailto:feyiobiang@yahoo.fr}{\color{blue}{\underline{\smash{feyiobiang@yahoo.fr}}}}\\[1mm]
$^2$LIBMA Laboratory, Department of Mathmatics, Faculty of Sciences Semlalia, Cadi Ayyad University, P.B.O. 2390 Marrakech, Morocco and Hassan II Academy of Sciences and Technologies, Rabat, Morocco\\
Email:
\href{mailto:ouknine@ucam.ac.ma}{\color{blue}{\underline{\smash{ouknine@ucam.ac.ma}}}}\\[1mm]
$^3$ URMI Laboratory, Département de Mathématiques et Informatique, Faculté des Sciences, Université des Sciences et Techniques de Masuku, Franceville, Gabon \\
\href{mailto:octavemoutsing-pro@yahoo.fr}{\color{blue}{\underline{\smash{octavemoutsing-pro@yahoo.fr}}}}\\[1mm]
$^4$ Department of Mathematical Sciences and Research Institute of Mathematics of Seoul National University, \\599 Gwanak-Ro, Gwanak-Gu, Seoul 08826, South Korea\\
Email:
\href{mailto:trutnau@snu.ac.kr}{\color{blue}{\underline{\smash{trutnau@snu.ac.kr}}}}\\[1mm]
\lin{5}{7}

 {  
 {\noindent{\large\bf\color{myaqua} Abstract}{\bf \\[3mm]
 \textup{
 This paper consists of two independent parts. In the first one, we  contribute to the study of the class $(\Sigma)$. For instance, we provide a new way to characterize stochastic processes of this class. We also present some new properties and solve the Bachelier equation. In the second part, we study the class of stochastic processes $\Sigma(H)$. This class was introduced in \cite{f} where from tools of the theory of martingales with respect to a signed measure of \cite{chav}, the authors provide a general framework and methods for dealing with processes of this class. In this work, after developing some new properties, we embed a non-atomic measure $\nu$ in $X$, a process of the class $\Sigma(H)$. More precisely, we find a stopping time $T<\infty$ such that the law of $X_{T}$ is $\nu$.
 }}}
 \\[4mm]
 {\noindent{\large\bf\color{myaqua} Keywords:}{\bf \\[3mm]
 Bachelier equation; Skorokhod embedding Problem; Class $\Sigma(H)$; class $(\Sigma)$; relative martingales; martingale with respect to a signed measure; Hardy-Littlewood function
}}\\[4mm]{\noindent{\large\bf\color{myaqua} MSC:}{\color{blue} 60G07; 60G20; 60G46; 60G48}}
\lin{3}{1}

\renewcommand{\headrulewidth}{0.5pt}
\renewcommand{\footrulewidth}{0pt}

 \pagestyle{fancy}
 \fancyfoot{}
 \fancyhead{} 
 \fancyhf{}
 \fancyhead[RO]{\leavevmode \put(-90,0){\color{myaqua}F. EYI-OBIANG et al} \boxx{15}{-10}{10}{50}{15} }
 \fancyfoot[C]{\leavevmode
 \put(-2.5,-3){\color{myaqua}\thepage}}

 \renewcommand{\headrule}{\hbox to\headwidth{\color{myaqua}\leaders\hrule height \headrulewidth\hfill}}
\section{Introduction}

{ \fontfamily{times}\selectfont
 \noindent 

 In this paper, we consider stochastic processes $(X_{t})_{t\geq0}$ of the form 
\begin{equation}\label{equ1}
	X_{t}=M_{t}+V_{t},
\end{equation}
where $V$ is an adapted continuous finite variation process such that $dV_{t}$ is carried by a closed optional set. Such processes have received much attention in probability theory. For instance,  the equation \eqref{equ1} is referred to as Skorokhod's reflection equation when $X\geq0$ is continuous and $V$ is increasing and continuous with $dV_{t}$ carried by the set $\{t\geq0: X_{t}=0\}$. It plays an important role in martingale theory. For instance: the family of Azéma-Yor martingales, the resolution of Skorokhod's embedding problem, the study of Brownian local times. It also plays a key role in the study of zeros of continuous martingales. A large class containing all the previously mentioned processes is the class $(\Sigma)$ whose  definition is given in section 2. This class was introduced by Yor in \cite{y1} but some of its main properties were further studied in \cite{pat,naj,naj1,naj2,naj3,nik,mult}.

Recently, the study of stochastic processes of the form of Equation \eqref{equ1} was generalized in the field of stochastic calculus for signed measures. That is, the stochastic calculus when the measure space is governed, not by a probability measure, but by a general measure that can take positive and negative values (signed measure). In particular, a new class of processes satisfying Equation \eqref{equ1} was introduced in \cite{f} where the authors provide a general framework and methods based on the tools of the martingale theory for signed measures developed by Ruiz de Chavez in \cite{chav}. This class is called, class $\Sigma(H)$ and we shall define it in Section 3.

The aim of this paper is to bring some contributions in frameworks of two above mentioned classes of stochastic processes. More precisely, the Section 2 is dedicated to the study of the class $(\Sigma)$ and Section 3 is reserved to the study of the class $\Sigma(H)$. In particular, in Section 2, we shall provide a new characterization of processes of class $(\Sigma)$. We give also a new solution of Bachelier equation. In Section 3, we prove the following estimates for processes of the class $\Sigma(H)$, generalizing a well known result for the pair $(X_{t},A_{t})$  and for random variable $A_{\infty}$ where $X$ is a positive submartingale of the class $(\Sigma)$ and $A$ its non-decreasing process:
\begin{equation}
\Pv(\exists t\geq\overline{g}, X_{t}>\varphi(A_{t}))=1-\exp\left(-\int_{0}^{\infty}{\frac{dz}{\varphi(z)}}\right),
\end{equation}
where $\varphi$ is a positive Borel function and $\overline{g}$ is an honest time defined in Section 3. And	
\begin{equation}
	\Pv(A_{\infty}>x)=\exp\left(-\int_{0}^{x}\frac{dz}{\lambda(z)}\right), x<b,
\end{equation}
where $\lambda(x)=\E[X_{\infty}|A_{\infty}=x]$ and $b\equiv inf\{u: \Pv(A_{\infty}\geq u)=0\}$. We use these estimates to solve the Skorokhod embedding problem on the class $\Sigma(H)$. More precisely, we embed a non atomic measure $\nu$ in a process $X$ of the class $\Sigma(H)$. 
}
\section{Some contributions to the study of the class \texorpdfstring{$(\Sigma)$}{sigma}}

{ \fontfamily{times}\selectfont
 \noindent 
 \subsection{Characterization and properties}
 Let us recall the definition of class $(\Sigma)$. Throughout we fix a filtered probability space $(\Omega,(\F_{t})_{t\geq0},\F,\Pv)$ satisfying the usual conditions.
 \begin{defn}
 We say that a stochastic process $X$ is of class $(\Sigma)$ if it decomposes as $X=M+V$, where
\begin{enumerate}
	\item $M$ is a càdlàg local martingale,
	\item $V$ is an adapted continuous finite variation process starting at 0,
	\item $\int_{0}^{t}{1_{\{X_{u}\neq0\}}dV_{u}}=0$ for all $t\geq0$.
\end{enumerate}   
 \end{defn}
 Recall that a stochastic process $X$ is said to be of class $(D)$ if $\{X_{\tau}:\tau\hspace{0.15cm} is\hspace{0.15cm} a\hspace{0.15cm} finite\hspace{0.15cm} stopping\hspace{0.15cm} time\}$ is uniformly integrable. We shall say that $X$ is of class $(\Sigma D)$ if $X$ is of class $(\Sigma)$ and of class $(D)$.
 
Nikeghbali has established  a martingale characterization theorem which permits to characterize the submartingales of class $(\Sigma)$. We recall it in Theorem \ref{mart}. In what follows, we provide a new way to characterize the processes of class $(\Sigma)$.  
 
Throughout, for any process $X$, we shall denote $g_{t}=\sup\{s\leq t: X_{s}=0\}$.
 \begin{theorem}\label{abs}
 Let $X$ be a continuous process which vanishes at zero. Then,
 $$X\in(\Sigma) \Leftrightarrow |X|\in(\Sigma).$$
 \end{theorem}
 \begin{proof}
 $\Rightarrow)$ Let $X=M+V$ be a stochastic process of class $(\Sigma)$. An application of Itô-Tanaka formula gives
 $$|X_{t}|=\int_{0}^{t}{{\rm sgn}(X_{s})dX_{s}}+L_{t}^{0}$$
 where, $L_{t}^{0}$ is the local time of $X$ at zero. This implies that
 $$|X_{t}|=\int_{0}^{t}{{\rm sgn}(X_{s})dM_{s}}+\int_{0}^{t}{{\rm sgn}(X_{s})dV_{s}}+L_{t}^{0}.$$
 But,
 $$\int_{0}^{t}{{\rm sgn}(X_{s})dV_{s}}=0$$
 since $dV_{t}$ is carried by $\{t\geq0: X_{t}=0\}=\{t\geq0: {\rm sgn}(X_{t})=0\}$. Hence, it follows that:
 $$|X_{t}|=\int_{0}^{t}{{\rm sgn}(X_{s})dM_{s}}+L_{t}^{0}.$$
 Therefore, $|X|\in(\Sigma)$ since $L_{t}^{0}$ is an increasing and continuous process such that $dL_{t}^{0}$ is carried by $\{t\geq0: X_{t}=0\}$.\\
 $\Leftarrow)$ Let us set 
 $$K_{t}=\lim_{s\searrow t}\inf{(1_{X_{s}>0}-1_{X_{s}<0})}.$$
 $K$ is a progressive and bounded process and $^{p}K_{\cdot}$ denotes its predictable projection. Hence, we have from the balayage formula in the progressive case (see \cite{bal} or Proposition 2.1 of \cite{siam}) that
 $$K_{g_{t}}|X_{t}|=\int_{0}^{t}{^{p}(K_{g_{s}})d|X_{s}|}+R_{t}$$
 where $R_{t}$ is a bounded variations process and $dR_{t}$ is carried by $\{t\geq0: X_{t}=0\}$.
 Consequently, 
 $$K_{g_{t}}|X_{t}|\in(\Sigma).$$
 But, 
 $$K_{g_{t}}|X_{t}|=X_{t}.$$
 This completes the proof.
 \end{proof}
 \begin{rem}
 For any positive submartingale $X$ of class $(\Sigma)$, there exists a stochastic process $Y$ of class $(\Sigma)$ such that $X=|Y|$. 
 \end{rem}
 
 In what follows, we present some corollaries of Theorem \ref{abs} above. We shall denote $g=\sup\{t\geq0: X_{t}=0\}$. 
  \begin{theorem}
Let $X$ be a continuous process of class $(\Sigma)$ and $L$ its local time in zero. Let $f:\R_{+}\to\R_{+}$ be a locally bounded Borel function such that $f(L_{t})|X_{t}|$ is of class $(D)$ and $f(L_{t})|X_{t}|\to1$ almost surely when $t\to\infty$. Denote $F(x)=\int_{0}^{x}{f(s)ds}$ and $$K_{t}=\lim_{s\searrow t}\inf{(1_{\{X_{s}>0\}}-1_{\{X_{s}<0\}})}.$$ Then, the following holds.
\begin{enumerate}
	\item If $F(\infty)<\infty$, then $(K_{g_{t}})_{t\geq0}$ is a process of class $(\Sigma)$. More precisely, $K_{g_{t}}=f(0)X_{t}$ for all $t\geq0$.
	\item If $F(\infty)=\infty$, then $g<\infty$ and for every stopping time $T$, one has:
\begin{equation}
	f(L_{T})X_{T}=\E[K_{g}1_{\{g\leq T\}}|\mathcal{F}_{T}].
\end{equation}
\end{enumerate}
\end{theorem}
\begin{proof}
From Theorem \ref{abs} , $|X|$ is a nonnegative process of class $(\Sigma)$. Then, according to Theorem 3.8 of \cite{pat}
\begin{enumerate}
	\item If $F(\infty)<\infty$, then 
\begin{equation}\label{kg1}
	f(0)|X_{t}|=1,
\end{equation}
for all $t\geq0$.
	\item If $F(\infty)=\infty$, then $g<\infty$ and for every stopping time $T$
\begin{equation}\label{kg2}
	f(L_{T})|X_{T}|=\Pv[g\leq T|\mathcal{F}_{T}].
\end{equation}
\end{enumerate}
Hence, multiplying \eqref{kg1} and \eqref{kg2} respectively by $K_{g_{t}}$ and $K_{g_{T}}$
\begin{enumerate}
	\item If $F(\infty)<\infty$, then 
\begin{equation}
	f(0)K_{g_{t}}|X_{t}|=K_{g_{t}},
\end{equation}
for all $t\geq0$.
	\item If $F(\infty)=\infty$, then $g<\infty$ and for every stopping time $T$
\begin{equation}
	f(L_{T})K_{g_{T}}|X_{T}|=K_{g_{T}}\Pv[g\leq T|\mathcal{F}_{T}].
\end{equation}
\end{enumerate}
But $K_{g_{t}}|X_{t}|=X_{t}$. Furthermore,
$$K_{g_{T}}\Pv[g\leq T|\mathcal{F}_{T}]=\E[K_{g_{T}}1_{\{g\leq T\}}|\mathcal{F}_{T}]=\E[K_{g}1_{\{g\leq T\}}|\mathcal{F}_{T}].$$
Consequently, the theorem is proved.
\end{proof}
 Now, we shall give some new properties of the class $(\Sigma)$.
 \begin{theorem}\label{1}
 Let $X=M+V$ be a stochastic process of class $(\Sigma)$. Then, for any bounded predictable process $K$, 
 $$(K_{g_{t}}\cdot X_{t})\in (\Sigma).$$
 \end{theorem}
 \begin{proof}
 An application of the balayage formula gives
 $$K_{g_{t}}X_{t}=\int_{0}^{t}{K_{g_{s}}dX_{s}}=\int_{0}^{t}{K_{g_{s}}dM_{s}}+\int_{0}^{t}{K_{g_{s}}dV_{s}}$$
 and the assertion follows.
 \end{proof}
 
 \begin{coro}
 Let $X=M+V$ be a stochastic process of class $(\Sigma)$. Then, for all measurable and bounded functions $f$,
 $$(f(V_{t})\cdot X_{t})\in (\Sigma).$$
 \end{coro}
 \begin{proof}
 It suffices to see that $V_{t}=V_{g_{t}}$. 
 \end{proof}
 
 The authors of \cite{pat} showed that under some assumptions, a submartingale of class $(\Sigma)$ can be written as $X_{t}=\E[X_{\infty}1_{\{g\leq t\}}|\mathcal{F}_{t}]$, where $g=\sup\{t\geq0: X_{t}=0\}$ and $\lim_{t\to \infty}{X_{t}}=X_{\infty}$. Here, we present a new corollary of this result.
 \begin{coro}
 Let $X$ be a submartingale of class $(\Sigma D)$ and $K$ be a bounded predictable process. Then, there exists an integrable random variable $X_{\infty}$ such that $\lim_{t\to \infty}{X_{t}}=X_{\infty}$ almost everywhere on the set $\{g<\infty\}$ and
\begin{equation}
	 K_{g_{t}} X_{t}=\E[K_{g}X_{\infty}1_{\{g<t\}}|\mathcal{F}_{t}].
\end{equation}
 \end{coro}
 \begin{proof}
 Since $X$ is a stochastic process of class $(D)$ and $K$ is a bounded process, we obtain from Theorem \ref{1} that $K_{g_{t}} X_{t}$ is a submartingale of class $(\Sigma)$. It follows that $K_{g_{t}} X_{t}$ is again a process of class $(D)$. Then, Theorem 3.1 and Remark 3.7 of \cite{pat} permit to conclude the proof.
 \end{proof}
 
 \begin{theorem}
 Let $k>0$ be a constant and $M$ a positive local martingale such that
 $$\lim_{t\to +\infty}{M_{t}=0}\hspace{1cm} almost\hspace{0.1cm}surely.$$
Set $g^{k}=\sup\{t\geq0: M_{t}=k\}$. Then,
 \begin{equation}
	\Pv[g^{k}>t|\mathcal{F}_{t}]=1\wedge\left(\frac{M_{t}}{k}\right).
 \end{equation}
 \end{theorem}
 \begin{proof}
 Note that $(k-M_{t})^{+}\in(\Sigma D)$ and $k=(k-M_{\infty})^{+}$. Hence,
 $$(k-M_{t})^{+}=\E[k\cdot1_{\{g^{k}<t\}}|\mathcal{F}_{t}]$$
 $$\hspace{1.4cm}=k\cdot\Pv[g^{k}<t|\mathcal{F}_{t}].$$
 It follows that
 $$\left(1-\frac{M_{t}}{k}\right)^{+}=\Pv[g^{k}<t|\mathcal{F}_{t}],$$
 thus
 $$\Pv[g^{k}>t|\mathcal{F}_{t}]=1-\left(1-\frac{M_{t}}{k}\right)^{+}=1\wedge\left(\frac{M_{t}}{k}\right).$$
 \end{proof}
 
 \subsection{The Bachelier equation and the class \texorpdfstring{$(\Sigma)$}{sigma}}
 Now, we shall recall a result of \cite{nik}. It is the following martingale characterization for the submartingales of class $(\Sigma)$.
 \begin{theorem}\label{mart}
 Let $X=M+V$ be a local submartingale, where $M$ is a local martingale and $V$ is the non-decreasing part of $X$ in the Doob-Meyer decomposition. The following are equivalent:
\begin{enumerate}
	\item The local submartingale $X$ is of class $(\Sigma)$.
	\item For every locally bounded Borel function $f$, and $F(x)\equiv\int_{0}^{x}{f(z)dz}$, the process
	$$W_{t}^{F}(X)=F(V_{t})-f(V_{t})X_{t}$$
	is a local martingale.
\end{enumerate}
 \end{theorem}
 
 Throughout the rest of this section, we focus our interest to the study of the local martingale $W_{t}^{F}(X)$. More precisely, we shall show that this local martingale solves the following Bachelier equation:
 
\begin{equation}
	dY_{t}=\varphi(\overline{Y}_{t})dN_{t}, \hspace{1cm} Y_{0}=0
\end{equation}
where $N=-M$, $M$ being the martingale part of $X$ and $\overline{Y}_{t}=\sup_{s\leq t}{Y_{s}}$. Let us begin with deriving some properties of the process $W_{t}^{F}(X)$.

 \begin{prop}\label{sup}
 Let $X=M+A$ be a positive local submartingale of class $(\Sigma)$, $f$ and $g$ two locally bounded and positive Borel functions. Let $G(x)=\int_{0}^{x}{g(z)dz}$ and $F(x)=\int_{0}^{x}{f(z)dz}$. Then, one has the following
\begin{equation}
\overline{W}_{t}^{F}(X)=F(A_{t});
\end{equation}
and
\begin{equation}
W^{G}_{t}(f(A_{\cdot})X_{\cdot})=W^{G\circ F}_{t}(X).
\end{equation}
 
\end{prop}
\begin{proof}
We have for all $t\geq0$,
$$F(A_{t})=(F(A_{t})-f(A_{t})X_{t})+f(A_{t})X_{t}.$$
Hence,
$$F(A_{t})=W_{t}^{F}(X)+f(A_{t})X_{t}.$$
Since for all $t\geq0$, $f(A_{t})X_{t}\geq0$ because $f$ is a positive function and $X$ is a nonnegative stochastic process, it follows that for all $t\geq0$,
$$W_{t}^{F}(X)\leq F(A_{t}).$$
Moreover, if $t$ is  a point of increase of $A_{t}$, one has:
$$W_{t}^{F}(X)=F(A_{t}).$$
It follows that 
$$\overline{W}_{t}^{F}(X)=F(A_{t}).$$
Note that $f(A_{t})X_{t}$ is a submartingale of class $(\Sigma)$ with non-decreasing part $F(A_{t})$. Hence, we have
$$W^{G}_{t}(f(A_{\cdot})X_{\cdot})=G(F(A_{t}))-g(F(A_{t}))f(A_{t})X_{t}.$$
Since, the derivative of $G\circ F$ is $(G\circ F)^{'}(x)=f(x)g(F(x))$, we obtain 
$$W^{G}_{t}(f(A_{\cdot})X_{\cdot})=G(F(A_{t}))-(G\circ F)^{'}(A_{t})X_{t}.$$
Therefore,
$$W^{G}_{t}(f(A_{\cdot})X_{\cdot})=W^{G\circ F}_{t}(X).$$
\end{proof} 

\begin{coro}\label{csup}
Let $X=M+A$ be a submartingale of class $(\Sigma)$. Let us put $N=-M$. Then, 
$$A_{t}=\sup_{s\leq t}{N_{s}}.$$
\end{coro}
\begin{proof}
Let $f$ be a positive, locally bounded Borel function. Let $F(x)=\int_{0}^{x}{f(z)dz}$. From Proposition \ref{sup}, we get
$$W^{F^{-1}}_{t}(f(A_{\cdot})X_{\cdot})=N_{t}.$$
Hence, by Proposition \ref{sup}, again
$$\sup_{s\leq t}{N_{s}}=\overline{W}^{F^{-1}}_{t}(f(A_{\cdot})X_{\cdot})=F^{-1}(F(A_{t})).$$
Consequently,
$$A_{t}=\sup_{s\leq t}{N_{s}}.$$
\end{proof}

\begin{theorem}
Let $X=M+A$ be a submartingale of class $(\Sigma)$ and $N_{t}=-M_{t}$. Consider a positive Borel function $\varphi:[0,+\infty)\longrightarrow(0,+\infty)$, such that $\frac{1}{\varphi}$ is locally integrable. Let $V(y)=\int_{0}^{y}\frac{ds}{\varphi(s)}$, and $U$ be its inverse defined on $(0,V(\infty))$. Then the stochastic process $Y_{t}=W^{U}_{t}(X)$, $\forall t< \tau^{\infty}_{N}$ is a strong, pathwise unique solution of the Bachelier equation
\begin{equation}
	dY_{t}=\varphi(\overline{Y}_{t})dN_{t}, \hspace{1cm} Y_{0}=0.
\end{equation}
\end{theorem}
\begin{proof}
It follows from Corollary \ref{csup} that $N$ is a max-continuous local martingale and 
$$X_{t}=\overline{N}_{t}-N_{t}.$$
Then, it follows thanks to Definition 2.1 of \cite{AY3} that $W^{U}_{t}(X)$ is an Azema-Yor martingale. Therefore, we obtain by Theorem 3.1 of \cite{AY3} that $W^{U}_{t}(X)$ is a strong, pathwise unique, max-continuous solution for $t< \tau^{\infty}_{N}$.
\end{proof}
 }

\section{Some contributions to the study of the class \texorpdfstring{$\Sigma(H)$}{sigma(H)}}

{ \fontfamily{times}\selectfont
 \noindent 
\subsection{Preliminaries and notations} 
We start by giving some notations which will be used in this section. Consider a measure space $(\Omega, \mathcal{F}_{\infty}, \Qv)$, where $\Qv$ is a bounded signed measure. Let $\Pv$ be a probability measure on $\mathcal{F}_{\infty}$ such that $\Qv\ll\Pv$. We shall throughout use the following notations:
\begin{itemize}
	\item $D_{t}=\frac{d\Qv}{d\Pv}|_{\mathcal{F}_{t}}$ where $(\mathcal{F})_{t\geq0}$ is a right continuous filtration completed with respect to $\Pv$ such that $\mathcal{F}_{\infty}=\vee_{t}\mathcal{F}_{t}$. Note that $D$ is a uniformly integrable $\Pv$- martingale (see Beghdadi-Sakrani \cite{sak}).
	\item $H=\{t: D_{t}=0\}$;
	\item $g=\sup{H}$; $\overline{g}=0\vee g$;
	\item The smallest right continuous filtration containing $(\mathcal{F}_{t})$ for which $\overline{g}$ is a stopping time will be denoted by $(\mathcal{F}^{g}_{t})$.
 Then, the filtration $(\mathcal{F}^{g}_{\overline{g}+t})$ is well defined and will be denoted $(\mathcal{F}_{t+g})$.
	\item $\Pv^{'}=\frac{|D_{\infty}|}{\E(|D_{\infty}|)}\Pv$.
\end{itemize}

In this work, we shall assume that $D$ is a continuous process which satisfies $\Pv(D_{\infty}=0)=0$ (i.e $\overline{g}<\infty$ almost surely). A martingale with respect to a signed measure $\Qv$ was defined by Ruiz de Chavez \cite{chav} as follows:
\begin{defn}\label{chav}
We say that, an $(\mathcal{F}_{t})_{t\geq0}$- adapted process $X$ is a $(\Qv,\Pv)$- martingale if:
\begin{enumerate}
	\item $X$ is a $\Pv$- semimartingale.
	\item $XD$ is a $\Pv$- martingale.
\end{enumerate}
\end{defn}

Now, recall the definition of stochastic processes of class $\Sigma(H)$.
\begin{defn}
Let $X$ be a nonnegative $\Pv$- semi-martingale, which decomposes as:
$$X_{t}=M_{t}+A_{t}.$$
We say that $X$ is of class $\Sigma(H)$, if:
\begin{enumerate}
	\item $M$ is a càdlàg $(\Qv,\Pv)$- local martingale with $M_{0}=0$;
	\item $A$ is a continuous non-decreasing process with $A_{0}=0$;
	\item the measure $(dA_{t})$ is carried by the set $\{t: X_{t}=0\}\cup H$.
\end{enumerate}
\end{defn}

In what follows, we show how a stochastic process of class $\Sigma(H)$  can be transformed into a process of class $(\Sigma)$. 
\begin{prop}\label{vu1}
Let $X=M+A$ be a process of  class $\Sigma(H)$  such that $X_{\overline{g}}=0$. Then, the process $X_{\cdot+\overline{g}}$ is a $\Pv^{'}$- submartingale of  class $(\Sigma)$ and its non-decreasing process is $A^{'}_{\cdot}=A_{\cdot+\overline{g}}-A_{\overline{g}}$.
\end{prop}
\begin{proof}
Let $X=M+A$ be a stochastic process which satisfies the assumptions of Proposition \ref{vu1}. It follows that
$$X_{t+\overline{g}}=X_{t+\overline{g}}-X_{\overline{g}}$$
since $X_{\overline{g}}=0$. Hence, we obtain 
$$X_{t+\overline{g}}=(M_{t+\overline{g}}-M_{\overline{g}})+(A_{t+\overline{g}}-A_{\overline{g}}).$$
Since $M$ is a $((\Qv,\Pv),(\F_{t})_{t\geq0})$- local martingale, it follows by the Quotient Theorem in \cite{1} that $M_{\cdot+\overline{g}}$ is a\\ $(\Pv^{'},(\F_{t+g})_{t\geq0}$- local martingale. This implies that $(M_{\cdot+\overline{g}}-M_{\overline{g}})$ is also a $(\Pv^{'},(\F_{t+g})_{t\geq0}$- local martingale. Furthermore 
$$A^{'}=A_{\cdot+\overline{g}}-A_{\overline{g}}$$
is an increasing and continuous process, $A^{'}_{0}=0$ and
$$\int{1_{\{X_{t+\overline{g}}\neq0\}}dA^{'}_{t}}=\int{1_{\{X_{t+\overline{g}}\neq0\}}dA_{t+\overline{g}}}.$$
But for any $t\geq0$ 
$$X_{t+\overline{g}}\neq0\Longleftrightarrow D_{t+\overline{g}}X_{t+\overline{g}}\neq0.$$
Hence
$$\int{1_{\{X_{t+\overline{g}}\neq0\}}dA^{'}_{t}}=\int{1_{\{D_{t+\overline{g}}X_{t+\overline{g}}\neq0\}}dA_{t+\overline{g}}}=0.$$
Therefore, $dA^{'}_{t}$ is carried by $\{t\geq0; X_{t+\overline{g}}=0\}$. Consequently, $X_{\cdot+\overline{g}}$ is a $\Pv{'}$- submartingale of class $(\Sigma)$.
\end{proof}

}
\subsection{Some estimates and distributions for \texorpdfstring{$(X_{t},A_{t})$}{(X(t),A(t))} and \texorpdfstring{$A_{\infty}$}{A(inf)}}
\label{sec:MLE}

{ \fontfamily{times}\selectfont
 \noindent
 
 The family of stopping times $T_{\varphi}$ associated with a nonnegative Borel function $\varphi$ is defined by
$$T_{\varphi}=\inf\{t\geq0: X_{t}\geq\varphi(A_{t})\}.$$
These stopping times play an important role in the resolution by Azéma and Yor \cite{2} and Obl\'oj and Yor \cite{15} of the Skorokhod embedding problem for the
Brownian Motion and the age of Brownian excursions. Some special subclasses of this family are also studied in \cite{6,9,nik}. One natural and important question is whether the stopping time $T_{\varphi}$ is almost surely finite or not. The aim of this section is to answer to this question when $X$ is a stochastic process of class $\Sigma(H)$  and $A$ is its non-decreasing part.

 \begin{theorem}\label{thsko}
Let $X=M+A$ be a stochastic process of class $\Sigma(H)$ with only negative jumps, such that $X_{\overline{g}}=0$ and $A_{\infty}=\infty$. Let $(\tau_{u})$ be the right continuous inverse of $A^{'}_{t}=(A_{t+\overline{g}}-A_{\overline{g}})$, i.e.
$$\tau_{u}\equiv \inf\{t\geq0; A^{'}_{t}>u\}.$$
Let $\varphi:\R_{+}\rightarrow\R_{+}$ be a Borel function. Then, we have the following estimates:
\begin{equation}
\Pv^{'}(\exists t\geq\overline{g}, X_{t}>\varphi(A_{t}-A_{\overline{g}}))=1-\exp\left(-\int_{0}^{\infty}{\frac{dz}{\varphi(z)}}\right)
\end{equation}
and
\begin{equation}
\Pv^{'}(\exists t \in [\overline{g},\overline{g}+\tau_{u}], X_{t}>\varphi(A_{t}-A_{\overline{g}}))=1-\exp\left(-\int_{0}^{u}{\frac{dz}{\varphi(z)}}\right).    
\end{equation}
\end{theorem}
\begin{proof}
Let $X$ be a $(\Pv,(\F_{t})_{t\geq0})$- semi-martingale satisfying the assumptions of the above theorem. First, observe that $X_{\cdot+\overline{g}}$ can be written as
$$X_{t+\overline{g}}=X_{t+\overline{g}}-X_{\overline{g}}$$
since $X_{\overline{g}}=0$. This implies that
$$X_{t+\overline{g}}=M^{'}_{t}+A^{'}_{t}$$
where, $M^{'}_{t}=(M_{t+\overline{g}}-M_{\overline{g}})$ and $A^{'}_{t}=(A_{t+\overline{g}}-A_{\overline{g}})$. But $(D_{t}M_{t})_{t\geq0}$ is a $(\Pv,(\F_{t})_{\geq0})$- local martingale. An application of the Quotient Theorem in  \cite{1} permits to see that $M_{\cdot+\overline{g}}$ is a $(\Pv^{'},(\F_{t+g})_{\geq0})$- local martingale. It follows that $M^{'}$ is also a $(\Pv^{'},(\F_{t+g})_{\geq0})$- local martingale which vanishes at zero. We can also see that the process $A^{'}$ is a non-decreasing and continuous process, $A^{'}_{0}=0$ and $dA^{'}_{t}=dA_{t+\overline{g}}$ is carried by  $\{t\geq0; X_{t+\overline{g}}=0\}$.

We claim that for all locally bounded Borel functions $f$, the process
$$f(A^{'}_{\cdot})X_{\cdot+\overline{g}}-\int_{0}^{{A^{'}}_{\cdot}}{f(z)dz}$$
is a local $(\Pv^{'},(\F_{t+g})_{\geq0})$- martingale. Indeed, let first $f$ in $\mathcal{C}^{1}(\R)$. Integration by parts gives 
$$f(A^{'}_{t})X_{t+\overline{g}}=\int_{0}^{t}{f(A^{'}_{s})dX_{s+\overline{g}}}+\int_{0}^{t}{f^{'}(A^{'}_{s})X_{s+\overline{g}}dA^{'}_{s}}.$$
But $\int_{0}^{t}{f^{'}(A^{'}_{s})X_{s+\overline{g}}dA^{'}_{s}}=0$, since $dA^{'}_{t}=dA_{t+\overline{g}}$ is carried by $\{t\geq0; X_{t+\overline{g}}=0\}$. Hence
$$f(A^{'}_{t})X_{t+\overline{g}}=\int_{0}^{t}{f(A^{'}_{s})dM^{'}_{s}}+\int_{0}^{t}{f(A^{'}_{s})dA^{'}_{s}}.$$
In particular
$$f(A^{'}_{t})X_{t+\overline{g}}-\int_{0}^{t}{f(A^{'}_{s})dA^{'}_{s}}=\int_{0}^{t}{f(A^{'}_{s})dM^{'}_{s}}$$
is a local $\Pv^{'}$- martingale. The general case when $f$ is only assumed to be locally bounded follows from a monotone class argument. We first note that for any Borel function $\varphi$ we can always assume that $\frac{1}{\varphi}$ is bounded and integrable (see the proof of [Theorem 3.2,\cite{nik}]).

For $F(x)=1-\exp{\left( -\int_{x}^{+\infty}\frac{dz}{\varphi(z)}\right)}$ its Lebesgue derivative $f$ is given by 
$$f(x)=-\frac{1}{\varphi(x)}\exp{\left( -\int_{x}^{+\infty}\frac{dz}{\varphi(z)}\right)}=-\frac{1}{\varphi(x)}(1-F(x)).$$
It follows that $$(U_{t}\equiv F(A^{'}_{t})-f(A^{'}_{t})X_{t+\overline{g}}),$$ which is also equal to $F(A^{'}_{t})+\frac{X_{t+\overline{g}}}{\varphi(A^{'}_{t})}(1-F(A^{'}_{t}))$ is a positive local $(\Pv^{'},(\F_{t+g})_{\geq0})$- martingale. Now,  put $S_{t}=\sup_{s\leq t}{U_{s}}$. Then $S$ is continuous since $U$ has only negative jumps. Observe also that
$$U_{0}=F(A^{'}_{0})=1-\exp{\left( -\int_{0}^{+\infty}\frac{dz}{\varphi(z)}\right)}$$
and that $U_{t}$ converges $\Pv^{'}$- almost surely as $t\to+\infty$ since $U$ is a positive local $\Pv^{'}$- martingale. 

Let us now consider
$$U_{\tau_{u}}=F(u)-f(u)X_{\tau_{u}+\overline{g}}.$$
Since $dA^{'}_{t}$ is carried by the zeros of $X_{\cdot+\overline{g}}$ and since $\tau_{u}$ corresponds to an increase time of $A^{'}$, we have $X_{\tau_{u}+\overline{g}}=0$. Consequently,
$$\lim_{u\to}{U_{\tau_{u}}}=\lim_{u\to+\infty}{F(u)}=0$$
and so
$$\lim_{u\to+\infty}{U_{u}}=0.$$

Now let us note that if for a given $t_{0}<+\infty$, we have $X_{t_{0}+\overline{g}}>\varphi(A^{'}_{t})$, then we must have
$$U_{t_{0}}>F(A^{'}_{t_{0}})-f(A^{'}_{t_{0}})\varphi(A^{'}_{t_{0}})=1$$
and hence we deduce that
$$\Pv^{'}(\exists t\geq0, X_{t+\overline{g}}>\varphi(A^{'}_{t}))=\Pv^{'}\left(\sup_{t\geq0}{\frac{U_{t}}{U_{0}}}>\frac{1}{U_{0}}\right).$$
But an application of Doob's maximal identity \cite{doob} permits us to see that
$$\Pv^{'}\left(\sup_{t\geq0}{\frac{U_{t}}{U_{0}}}>\frac{1}{U_{0}}\right)=U_{0}.$$
Consequently,
$$\Pv^{'}(\exists t\geq0, X_{t+\overline{g}}>\varphi(A^{'}_{t}))=1-\exp{\left( -\int_{0}^{+\infty}\frac{dz}{\varphi(z)}\right)}.$$
This implies that
$$\Pv^{'}(\exists t\geq0, X_{t+\overline{g}}>\varphi(A_{t+\overline{g}}-A_{\overline{g}}))=1-\exp{\left( -\int_{0}^{+\infty}\frac{dz}{\varphi(z)}\right)}$$
and so
$$\Pv^{'}(\exists t\geq \overline{g}, X_{t}>\varphi(A_{t}-A_{\overline{g}}))=1-\exp{\left( -\int_{0}^{+\infty}\frac{dz}{\varphi(z)}\right)}.$$
As in proof of Theorem 3.2 of A.Nikeghbali \cite{nik}, it is enough to replace $\varphi$ by the function $\varphi_{u}$ defined as
$$\varphi_{u}(x)= \left \lbrace \begin{array}{l}
                      \varphi(x)$, $if$ $x<u,\\
                      \infty$ $otherwise,
                  \end{array} \right.$$
to obtain the second identity of the theorem.
\end{proof}

Now, we shall give some corollaries of Theorem \ref{thsko}.
\begin{coro}\label{sig}
Let $X$ be a stochastic process of class $\Sigma(H)$ with only negative jumps, such that $X_{\overline{g}}=A_{\overline{g}}=0$ and $A_{\infty}=\infty$. Let $(\tau_{u})$ and $\varphi$ be defined as in Theorem \ref{thsko}. Then, we have the following estimates:
\begin{equation}
\Pv^{'}(\exists t\geq\overline{g}, X_{t}>\varphi(A_{t}))=1-\exp\left(-\int_{0}^{\infty}{\frac{dz}{\varphi(z)}}\right)
\end{equation}
and
\begin{equation}
\Pv^{'}(\exists t \in [\overline{g},\overline{g}+\tau_{u}], X_{t}>\varphi(A_{t}))=1-\exp\left(-\int_{0}^{u}{\frac{dz}{\varphi(z)}}\right)    
\end{equation}
\end{coro}

The following corollary permits to achieve the aim announced at the beginning of this subsection.
\begin{coro}\label{cosko}
Let $X$ be a stochastic process of class $\Sigma(H)$ with only negative jumps, such that $X_{\overline{g}}=0$ and $\lim_{t\to+\infty}{A_{t}}=+\infty$. If
$$\int_{0}^{+\infty}{\frac{dx}{\varphi(x)}}=+\infty,$$
then the stopping time $T_{\varphi}$ is finite $\Pv^{'}$- almost surely. Moreover, if $T<\infty$ and if $\varphi$ is locally bounded, then
\begin{equation}\label{eqsko}
X_{T+\overline{g}}=\frac{1}{\varphi(A_{T+\overline{g}}-A_{\overline{g}})}.
\end{equation}
\end{coro}
\begin{proof}
First, recall that in this work we assume that $\overline{g}<\infty$. From the assumption 
$$\int_{0}^{+\infty}{\frac{dx}{\varphi(x)}}=+\infty,$$
we obtain by Theorem \ref{thsko} that
$$\Pv^{'}(\exists t\geq\overline{g}, X_{t}>\varphi(A_{t}-A_{\overline{g}}))=1,$$
hence
$$\Pv^{'}(\exists t\geq0, X_{t+\overline{g}}>\varphi(A_{t+\overline{g}}-A_{\overline{g}}))=1.$$
Then, the stopping time 
$$T^{g}_{\varphi}=\inf\{t\geq0: X_{t+\overline{g}}\geq\varphi(A_{t+\overline{g}}-A_{\overline{g}})\}$$
is $\Pv^{'}$- almost surely  finite. But we can see that
$$\overline{g}+T^{g}_{\varphi}=\inf\{\overline{g}+t\geq0: X_{\overline{g}+t}\geq\varphi(A_{\overline{g}+t}-A_{\overline{g}})\}.$$
Therefore,
$$T_{\varphi}=\overline{g}+T^{g}_{\varphi}.$$
Consequently, $T_{\varphi}$ is $\Pv{'}$- almost surely finite. 

Furthermore, we know by Proposition \ref{vu1} that $(X_{\overline{g}+t})_{t\geq0}$ is a $\Pv^{'}$- submartingale of class $(\Sigma)$. Hence, from Corollary 3.3 of \cite{nik}, we obtain the equality \eqref{eqsko}. This achieves the proof.

\end{proof}
 
Let us give now some results about the approximation of the law of the random variable $A_{\infty}$.
\begin{theorem}\label{Ainf}
Let $X$ be a stochastic process of class $\Sigma(H)$ and of class $(D)$ such that $X_{\overline{g}}=A_{\overline{g}}=0$, and let
$$\lambda(x)=\E^{'}[X_{\infty}|A_{\infty}=x]$$
($\E^{'}$ is the expectation with respect to $\Pv^{'}$). Assume that $\lambda(A_{\infty})\neq0$. Then for
$$b\equiv inf\{u: \Pv^{'}(A_{\infty}\geq u)=0\},$$
it holds
$$\Pv^{'}(A_{\infty}>x)=\exp\left(-\int_{0}^{x}\frac{dz}{\lambda(z)}\right), x<b.$$
\end{theorem}
\begin{proof}
Let $f$ be a bounded Borel function with compact support. From Corollary 2.2 of \cite{f}, it follows that
$$M_{t}=F(A_{t})-f(A_{t})X_{t}$$
is a uniformly integrable $(\Qv,\Pv)$- martingale. This implies that $(\widetilde{M}_{t}=M_{t+\overline{g}})_{t\geq0}$ is a $\Pv^{'}$- martingale. Furthermore, $\widetilde{M}_{0}=0$ since $X_{\overline{g}}=A_{\overline{g}}=0$. Then,
\begin{equation}
\E^{'}[F(A_{\infty})]=\E^{'}[X_{\infty}f(A_{\infty})].
\end{equation}
But
$$\E^{'}[X_{\infty}f(A_{\infty})]=\E^{'}[\E^{'}[X_{\infty}f(A_{\infty})|A_{\infty}]]$$
$$\hspace{2.6cm}=\E^{'}[f(A_{\infty})\E^{'}[X_{\infty}|A_{\infty}]].$$
Therefore,
\begin{equation}\label{e}
\E^{'}[F(A_{\infty})]=\E^{'}[f(A_{\infty})\lambda(A_{\infty})].
\end{equation}
Now, since $\lambda(A_{\infty})>0$, if $\nu(dx)$ denotes the law of $A_{\infty}$ and $\overline{\nu}(x)=\nu([x,\infty[), \eqref{e}$ implies
$$\int_{0}^{\infty}{dzf(z)\overline{\nu}(z)}=\int_{0}^{\infty}{\nu(dz)f(z)\lambda(z)}$$
and consequently
\begin{equation}\label{e1}
\overline{\nu}(z)dz=\lambda(z)\nu(dz).
\end{equation}
Recall that $b\equiv \inf\{u: \Pv^{'}(A_{\infty}\geq u)=0\}$; hence for $x<b$
$$\int_{0}^{x}{\frac{dz}{\lambda(z)}}=\int_{0}^{x}{\frac{\nu(dz)}{\overline{\nu}(z)}}\leq\frac{1}{x}<\infty$$
and integrating \eqref{e1} from 0 to $x$, $x<b$, it follows that
$$\overline{\nu}(x)=\exp\left(-\int_{0}^{x}{\frac{dz}{\lambda(z)}}\right)$$
and the result follows easily. 
\end{proof}

We present the following interesting corollary.

\begin{coro}\label{sko}
Let $X$ be a process of class $\Sigma(H)$ and of class $(D)$ such that $X_{\overline{g}}=A_{\overline{g}}=0$ and $\lim_{t\to+\infty}{X_{t}}=a>0$, $a.s$.
Then
$$\Pv^{'}(A_{\infty}>x)=\exp\left(-\frac{x}{a}\right).$$
\end{coro}
\begin{proof}
That is a consequence of Theorem \ref{Ainf}, with $\lambda(x)\equiv a$.
\end{proof}
 
}
 
\subsection{The Skorokhod embedding problem for non-atomic laws on \texorpdfstring{$\mathbb{R}_{+}$}{R+}}

{ \fontfamily{times}\selectfont
 \noindent 
 
 Now, we shall use results of Subsection \ref{sec:MLE} to embed a non-atomic probability measure $\nu$ in a stochastic process $X$ of class $\Sigma(H)$. More precisely, let $\nu$ be a probability measure on $\R_{+}$, which has no atoms, and let $X$ be a stochastic process of class $\Sigma(H)$ with only negative jumps, such that $X_{\overline{g}}=A_{\overline{g}}=0$ and $A_{\infty}=+\infty$. Our aim is to find a stopping time $T_{\nu}$ such that the law of $X_{T_{\nu}}$ is $\nu$. Note that $T_{\nu}$ coincides with the stopping time proposed by Ob\'oj and Yor \cite{15} when $X_{t}=B_{t}$ is a Brownian motion and used by Nikeghbali \cite{nik} when $X$ is a submartingale of class $(\Sigma)$.

In what follows, we write $\overline{\nu}(x)=\nu([x,\infty))$ for the tail of $\nu$, $a_{\nu}=\sup\{x\geq0: \overline{\nu}(x)=1\}$, and $b_{\nu}=\inf\{x\geq0: \overline{\nu}(x)=0\}$, $-\infty\leq a_{\nu}\leq b_{\nu}\leq\infty$, respectively, for the lower and upper bounds of the support of $\nu$. Now, introduce  the dual Hardy-Littlewood function of \cite{15}
$\Psi_{\nu}:[0,+\infty)\to[0,+\infty)$ through
$$\Psi_{\nu}(x)=\int_{[0,x]}{\frac{z}{\overline{\nu}(z)}d\nu(z)},\hspace{1cm}a_{\nu}\leq x\leq b_{\nu},$$
$\Psi_{\nu}(x)=0$ for $0\leq x\leq a_{\nu}$ and $\Psi_{\nu}(x)=\infty$ for $x\geq b_{\nu}$. The function $\Psi_{\nu}$ is continuous, increasing and so we can define its right continuous inverse
$$\varphi_{\nu}(z)=\inf\{x\geq0: \Psi_{\nu}(x)>z\},$$
which is strictly increasing. We are ready to state the main result of this paper:

\begin{theorem}
Let $X$ be a stochastic process of class $\Sigma(H)$ with only negative jumps, such that $X_{\overline{g}}=A_{\overline{g}}=0$ and $A_{\infty}=+\infty$. The stopping time
$$T_{\nu}=\inf\{t\geq0: X_{t}\geq\varphi_{\nu}(A_{t})\}$$
is $\Pv^{'}$- almost surely finite and solves the Skorokhod embedding problem for $X$, i.e. the law of $X_{T_{\nu}}$ is $\nu$. 
\end{theorem}
\begin{proof}
First, we note that (see \cite{15})
$$\int_{0}^{x}{\frac{dz}{\varphi_{\nu}(z)}}<\infty\hspace{1cm}for\hspace{1cm}0\leq x<b_{\nu}$$
and
$$\int_{0}^{+\infty}{\frac{dz}{\varphi_{\nu}(z)}}=+\infty.$$
Consequently, it follows from Corollary \ref{cosko} that $T_{\nu}<\infty$ $\Pv^{'}$- $a.s$. Moreover, we know from Proposition \ref{vu1} that $(X_{t+\overline{g}})_{t\geq0}$ is a submartingale of class $(\Sigma)$ and Corollary \ref{sig} permits to affirm that the stopping time $$T^{g}_{\nu}=\inf\{t\geq0:X_{t+\overline{g}}\geq\varphi(A_{t+\overline{g}})\}$$ 
is finite $\Pv^{'}$- a.s. Since $T_{\nu}=T^{g}_{\nu}+\overline{g}$, it follows by Remark 3.5 of \cite{nik} that $$X_{T_{\nu}}=\varphi_{\nu}(A_{T_{\nu}}).$$

Now, let $h:\R_{+}\to\R_{+}$ be a strictly decreasing function, locally bounded, and such that
$$\int_{0}^{+\infty}{h(z)dz}=\infty.$$
According to Proposition 2.1 of \cite{f}, $h(A_{\cdot})X_{\cdot}$ is again a stochastic process of class $\Sigma(H)$, and its increasing process is
$$\int_{0}^{A_{t}}{h(z)dz}\equiv H(A_{t}).$$

The stopping time
$$R_{h}=\inf\{t\geq0: h(A_{t})X_{t}=1\}$$
is finite $\Pv^{'}$- almost surely see Corollary \ref{cosko}. Since
$$\lim_{t\to+\infty}{H(A_{t})}=+\infty$$
and that $h(A_{\overline{g}})X_{\overline{g}}=H(A_{\overline{g}})=0$, it follows from Corollary \ref{sko} that $H(A_{R_{h}})$ is distributed as a random variable $e$ which follows the standard exponential law, and hence
$$A_{R_{h}}\stackrel{law}{=}H^{-1}(e).$$
Consequently,
\begin{align}\label{0sko}
X_{R_{h}}\stackrel{law}{=}\frac{1}{h(H^{-}(e))}\hspace{1cm}since\hspace{0.5cm}X_{R_{h}}=\frac{1}{h(A_{R_{h}})}.
\end{align}	

Now, we investigate the converse problem; that is, given a probability measure $\nu$, we want to find $h$ such that
$$X_{R_{h}}\stackrel{law}{=}\nu.$$
From \eqref{0sko}, we deduce 
$$\overline{\nu}(x)=\Pv^{'}\left(\frac{1}{h(H^{-}(e))}>x\right)$$
$$\hspace{1.2cm}=\Pv^{'}\left(H^{-1}(e)>h^{-1}\left(\frac{1}{x}\right)\right)$$
$$\hspace{0.98cm}=\Pv^{'}\left(e>H\left(h^{-1}\left(\frac{1}{x}\right)\right)\right)$$
$$\hspace{0.98cm}=\exp\left(-H\left(h^{-1}\left(\frac{1}{x}\right)\right)\right).$$
Now, differentiating the last equality yields
$$-d\overline{\nu}(x)=\overline{\nu}(x)\left[h\left(h^{-1}\left(\frac{1}{x}\right)\right)\right]d\left(h^{-1}\left(\frac{1}{x}\right)\right)$$
$$\hspace{-1cm}=\frac{\overline{\nu}(x)}{x}d\left(h^{-1}\left(\frac{1}{x}\right)\right).$$
Hence
$$d\left(h^{-1}\left(\frac{1}{x}\right)\right)=x\frac{d\nu(x)}{\overline{\nu}(x)}.$$
Consequently, we have
$$h^{-1}\left(\frac{1}{x}\right)=\int_{0}^{x}{\frac{z}{\overline{\nu(z)}}d\nu(z)}=\Psi_{\nu}(x).$$
But
$$R_{h}=\inf\{t\geq0: h(A_{t})X_{t}=1\}$$
$$\hspace{1.5cm}=\inf\left\{t\geq0: A_{t}= h^{-1}\left(\frac{1}{X_{t}}\right)\right\}$$
$$\hspace{1cm}=\inf\{t\geq0: A_{t}=\Psi_{\nu}(X_{t})\}$$
$$\hspace{1cm}=\inf\{t\geq0: \Psi_{\nu}(X_{t})\geq A_{t}\}$$
$$\hspace{1.5cm}=\inf\{t\geq0: X_{t}\geq \Psi^{-1}_{\nu}(A_{t})\}.$$
Then
$$R_{h}=\inf\{t\geq0: X_{t}\geq \varphi_{\nu}(A_{t})\}=T_{\nu}.$$
Therefore, the proof of the theorem follows easily.
\end{proof}
 
}











 

 {\color{myaqua}

}}

\end{document}